\documentclass[12pt]{amsart}
\usepackage{latexsym,fancyhdr,amssymb,color,amsmath,amsthm,graphicx,listings,comment}
\newtheorem{thm}{Theorem}       \newtheorem{propo}{Proposition}
\newtheorem{lemma}{Lemma}       \newtheorem{coro}{Corollary}
\usepackage[section]{placeins}      \let\paragraph\subsection
\def\B#1#2{{#1\choose #2}}

\title{Dehn-Sommerville from Gauss-Bonnet}
\author{Oliver Knill} \date{5/12/2019}
\address{Department of Mathematics \\ Harvard University \\ Cambridge, MA, 02138 }

\begin{document}

\begin{abstract}
We give a ``zero curvature" proof of Dehn-Sommerville for finite simple
graphs. It uses a parametrized Gauss-Bonnet formula telling that
the curvature of the valuation $G \to f_G(t)=1+f_0 t + \cdots + f_d t^{d+1}$ 
defined by the $f$-vector of $G$ is the anti-derivative $F$ of $f$ evaluated on the unit 
sphere $S(x)$. Gauss Bonnet is then parametrized, 
$f_G(t) = \sum_x F_{S(x)}(t)$, and holds for all Whitney simplicial complexes $G$. 
The Gauss-Bonnet formula $\chi(G)=\sum_x K(x)$ for Euler characteristic $\chi(G)$ 
is the special case $t=-1$. Dehn-Sommerville is equivalent to the reflection symmetry 
$f_G(t)+(-1)^d f_G(-1-t)=0$ which is equivalent to the same symmetry for $F$. 
Gauss-Bonnet therefore relates Dehn-Sommerville for $G$ with 
Dehn-Sommerville for the unit spheres $S(x)$, where it is a zero curvature condition.
A class $\mathcal{X}_d$ of complexes for which Dehn-Sommerville holds
is defined inductively by requiring $\chi(G)=1+(-1)^d$ and $S(x) \in \mathcal{X}_{d-1}$
for all $x$. It starts with $\mathcal{X}_{(-1)}=\{ \{\} \}$.
Examples are simplicial spheres, including homology spheres, 
any odd-dimensional discrete manifold, any even-dimensional discrete manifold with $\chi(G)=2$. 
It also contains non-orientable ones for which Poincar\'e-duality fails or stranger spaces like spaces 
where the unit spheres allow for two disjoint copies of manifolds with $\chi(G)=1$. 
Dehn-Sommerville is present in the Barycentric limit. 
It is a symmetry for the Perron-Frobenius eigenvector of the Barycentric refinement operator $A$.
The even eigenfunctions of $A^T$, the Barycentric Dehn-Sommerville functionals, vanish on $\mathcal{X}$
like $22 f_1 - 33 f_2 + 40 f_3 - 45f_4=0$ for all $4$-manifolds.
\end{abstract}
\maketitle

\section{Gauss Bonnet}

\paragraph{}
The category of {\bf finite abstract simplicial complexes} $G$ requires
only one axiom: $G$ is a set of non-empty sets closed under the operation of 
taking finite non-empty subsets. The {\bf $f$-vector} of $G$ is 
$f=(f_0,f_1, \dots ,f_d)$ of $G$, where $f_k$
is the set of sets in $G$ with $k+1$ elements. The {\bf $f$-function} of $G$ is 
$f_G(t) = 1+\sum_{k=0}^d f_k(G) t^k$. If $G$ is the Whitney complex of a graph $(V,E)$
the  {\bf unit sphere} $S(x)$ for $x \in V(G)$ is the unit sphere in that graph. 
For any $G$, let $F_G(t)=\int_0^t f_G(s) \; ds$ denote the anti-derivative of $f_G$. 
The curvature valuation to $f$ is the anti-derivative of $f$ evaluated on the unit sphere:

\begin{thm}[Gauss-Bonnet]
$f_G(t) = \sum_{x \in G} F_{S(x)}(t)$.
\end{thm}

\begin{proof}
If every $k$-simplex $y$ in $S(x)$ carries a charge $t^{k+1}$, then 
$f_G(t)$ is the total charge. Because every $k$-simplex $y$ in $S(x)$ defines a 
$(k+1)$-simplex $z$ in $G$, the simplex $z$ in $S(x)$ carries a charge $t^{k+2}$. 
It contains $(k+2)$ zero-dimensional points, which were simplices in $G$. 
Distributing the charge equally to the points, gives each a charge $t^{k+2}/(k+2)$. 
The curvature $F_{S(x)}(t)$ at $x$ adds up all the charges of the simplices attached
to $x$. There is code at the end allowing to experiment. 
\end{proof}

\paragraph{}
For $t=-1$, we get a classical Gauss-Bonnet statement
$$ \chi(G) = \sum_{x \in G} K(x) \; ,  $$
where $K(x)=F_{S(x)}(-1)$ is the Levitt curvature, the discrete analogue of
the {\bf Gauss-Bonnet-Chern curvature} in the continuum. An explicit formula for $K(x)$ 
with $f_{-1}=1$ is 
\begin{equation}
  K(x) = \sum_{k=-1}^{d} (-1)^k \frac{f_k(S(x))}{k+2}  \; . 
  \label{levitt}
\end{equation}
It appeared first in \cite{Levitt1992}. For the continuum proof, see \cite{Cycon}. 

\paragraph{}
By differentiation of Gauss-Bonnet with respect to $t$ we get:
$f_G'(t)=\sum_{x} f_{S(x)}(t)$.
For $t=-1$ in particular, this gives an identity seen in \cite{DehnSommerville}.
$f_G'(-1) = \sum_{x \in G} 1-\chi(S(x))$
which is a trace of the Green function operator $g=L^{-1}$ with $L_{xy}=1$ if $x \cap y \neq \emptyset$
and $L_{xy}=0$ else, where the diagonal entries are $g(x,x)=1-\chi(S(x))$. 

\section{Dehn-Sommerville symmetry}

\paragraph{}
Given a complex $G$, the {\bf $h$-function} $h_G(x) = (x-1)^d f_G(1/(x-1))$ is a polynomial
$h_G(x) = h_0 + h_1 x + \cdots + h_d x^d + h_{d+1} x^{d+1}$ 
defining a {\bf $h$-vector} $(h_0,h_1, \dots, h_{d+1})$. 
The {\bf Dehn-Sommerville relations} assert that the $h$-vector is {\bf palindromic}, 
meaning that $h_i=h_{d+1-i}$ for all $i=0, \dots, d+1$. 
Let us call a complex $G$ {\bf Dehn-Sommerville} if the Dehn-Sommerville relations hold for $G$. 

\paragraph{}
For example, for the {\bf icosahedron complex} $G$ generated by the triangles
$\{\{1,2,5\}$, $\{1,2,6\}$, $\{1,3,4\}$, $\{1,3,5\}$, $\{1,4,6\}$, $\{2,5,9\}$, $\{2,6,1\}$,
$\{2,9,10\}$, $\{3,4,8\}$, $\{3,5,11\}$, $\{3,8,11\}$, $\{4,6,12\}$, $\{4,8,12\}$,
$\{5,9,11\}$, $\{6,10,12\}$, $\{7,8,11\}$,$\{7,8,12\}$, $\{7,9,10\}$,$\{7,9,11\}$,
$\{7,10,12\}\}$, with $d=2$ and $f_G(t)=1+12t+30t^2+20 t^3$ we have
$h_G(t)=1+9t+9t^2+1$. The f-vector $(12,30,20)$ defined the h-vector $(1,9,9,1)$. The graph is
Dehn-Sommerville. The Whitney complex $G_1$ of the icosahedron has the $f$-vector $(62,180,120)$ and
$h$-vector $(1,59,59,1)$. As an other example, the M\"obius strip $G$ generated by 
$\{\{1,2,5\}$,$\{1,5,8\}$, $\{2,3,6\}$, $\{2,5,6\}$, $\{3,4,7\}$,
$\{3,6,7\}$, $\{4,5,8\}$, $\{4,7,8\}\}$,
with $f$-vector $(8,16,8)$ gets the $h$-vector $(-1,3,5,1)$ which is not palindromic. 
The M\"obius strip is not Dehn-Sommerville. Manifolds with boundary in general are not 
Dehn-Sommerville. We will see that cohomology, orientability etc are irrelevant. The only thing 
which matters is the Eulercharacteristic of the complex as well as whether the unit spheres and 
unit spheres of unit spheres etc are Dehn-Sommerville.

\paragraph{}
The following result holds for any simplicial complex:

\begin{thm}
The simplex generating function $f_G(t)$ of $G$ satisfies the
symmetry $f(t)+ (-1)^d f(-1-t)$ if and only if $G$ is Dehn-Sommerville. 
\end{thm}

\begin{proof}
The palindromic condition can be rephrased that the
roots of the $h$-function $h(t)=1+h_0 t + \cdots + h_d t^{d+1}$ are invariant under the 
involution $x \to 1/x$. This is equivalent that the roots of $f$ are invariant
under the involution $x \to -1-x$ and so to the symmetry $f(-1-t)= \pm f(t)$ 
for the $f$-function. 
\end{proof}

\paragraph{}
Dehn-Sommerville complexes must have the Euler characteristic of a $d$-sphere:

\begin{coro}
If $G$ is a complex with maximal dimension $d$ and $G$ satisfies Dehn-Sommerville, 
then $\chi(G) = 1+(-1)^d$. 
\end{coro}
\begin{proof}
We have $f_G(0)=1$. The symmetry tells $f(-1) = (-1)^d f(0)=(-1)^d$. But 
$f(-1)=1-\chi(G)$. 
\end{proof} 

\paragraph{}
For the zero-sphere $S^0 = \{ \{1\},\{2\} \}$  we have $f_G(t)=1+2t$ which satisfies $f_G(-t-1)=-f_G(t)$.
Since $f_{G+H}(t) = f_G(t) f_H(t)$, we immediately see that the {\bf suspension} $S_0 + G$ of a Dehn-Sommerville
complex is Dehn-Sommerville. More generally: 

\begin{coro}
If $G$ and $H$ are Dehn-Sommerville, then the join $G+H$ is Dehn-Sommerville.
\end{coro}

\paragraph{}
While the join of a $k$-sphere and a $l$-sphere is always a $k+l+1$-sphere, we in general
do not get discrete manifolds, if we take the join of two discrete manifolds. The join can
produce lots of examples of simplicial Dehn-Sommerville complexes which are not manifolds.

\paragraph{}
The {\bf Barycentric refinement} $G_1$ of a complex $G$ is the order complex of $G$. 
It is more intuitive to think of $G_1$ as the Whitney complex of the graph 
$\Gamma(G)$ defined by $G$. Barycentric refinements are always Whitney complexes 
of graphs. The sets in $G_1$ are the vertex sets of the complete sub-graphs of $\Gamma(G)$. 
The following statement can be reformulated algebraically as a commutation between 
two operations, the Barycentric refinement operation and the Dehn-Sommerville 
involution. But it is also a geometric statement: 

\begin{propo}
If $G$ is Dehn-Sommerville then its Barycentric refinement $G_1$ is Dehn-Sommerville.
\end{propo}
\begin{proof}
This can be proven by induction with respect to dimension.  Gauss-Bonnet implies that $G$ satisfies
Dehn-Sommerville if and only it has the right Euler characteristic $1+(-1)^d$ and all unit spheres
satisfy Dehn-Sommerville. The unit spheres of $G$ are either spheres or Barycentric refinements of
unit spheres of $G$. Both cases satisfy Dehn-Sommerville by induction.
\end{proof}

\paragraph{}
Let $A$ be the Barycentric refinement operator defined by $f(G_1)=A f(G)$. The matrix $A$ is a 
$(d+1) \times (d+1)$ upper triangular matrix and explicitly been given as $A_{ij} = {\rm Stirling}(i,j) i!$.  
Since all eigenvalues $\lambda_k=k!$ are distinct, the eigenvalues of $A$ are an eigenbasis of $A$
on the  vector space $V_d = \mathbb{R}^{d+1}$ which is isomorphic to the space 
$P_d$ of polynomials of degree less or equal to $d$. 
The isomorphism is given by $[a_0,a_1, \cdots, a_d] \to a_0 + a_1 t + \cdots + a_d t^d$. 
As an affine space, it is isomorphic to $1+a_0 t + a_1 t^2 + \cdots + a_d t^{d+1}$
which is the form of an f-function of a complex. 

\paragraph{}
The linear unitary reflection $T(f)(x) = f(-1-x)$ on polynomials defines an involution 
on $\mathbb{R}^{d+1}$. As an unitary reflection $T^2=Id$, it has the eigenvalues 
$1$ and $-1$ which by the spectral theorem of normal operators define an 
eigenbasis even so the algebraic multiplicities are larger than $1$ for $d>1$. 
In analogy to $\tilde{T}(f)(x)=f(-x)$, we can call eigenfunctions of $1$ 
{\bf even functions} and eigenfunctions of $-1$ {\bf odd functions}. 

\paragraph{}
Any eigenvector $V$ of $A^T$ defines a functional $\phi_V$ on complexes. 
Most functionals $\phi_V(G_n)$ explode when looking at Barycentric refinements $G_n$ of $G$.
There is just one functional $\phi_1$ which stays invariant and this is the Euler characteristic. 
Since the matrix $A$ is upper triangular and $A^T$ lower triangular, the eigenbasis diagonalizing $A$ is 
triangular too. We say an eigenvector is {\bf even} if it has an odd number of non-zero entries. 

\begin{lemma}
The eigenbasis of $A$ is also an eigenbasis of $T$: even eigenvectors are
eigenfunctions of $T$ to the eigenvalue $1$ and odd eigenvectors of $A$
are eigenfunctions of $T$ to the eigenvalue $-1$. 
\end{lemma}
\begin{proof}
As the linear operators $T$ and the Barycentric operation $A$ commute. They therefore
have the same eigenbasis. 
\end{proof} 

\paragraph{}
We know from linear algebra that the eigenvectors $V_k$ of $A^T$ and
the eigenvectors $W_k$ of $A$ have the property that $V_k W_l = c_{kl} \delta_{k,l}$ 
meaning that if they are normalized, then they define dual coordinate systems. 
We think about eigenvectors of $A^T$ as functionals. Functionals in the even eigenspace of $T$
are zero on even functions etc. This gives us convenient Dehn-Sommerville invariants: 

\begin{coro}
For even $d$, the even eigenvectors of $A^T$ and
for odd $d$, the odd eigenvectors of $A^T$ define functionals 
which are zero on the class $\mathcal{X}_d$.
\end{coro} 

\paragraph{}
This was Theorem~(1) in \cite{valuation}, where already the idea of proving
Dehn-Sommerville via curvature has appeared and multi-variate versions of 
Dehn-Sommerville were given, answering a open problem of Gruenbaum \cite{Gruenbaum1970} from 1970. 
The current approach is much simpler. In multi-dimensions, the Dehn-Sommerville symmetry
just has to hold for each of the variables appearing in the simplex generating function 
$f(t_1, \cdots ,f_m)$. The proof in higher dimensions is identical using Gauss-Bonnet. 

\paragraph{}
For the next part, we assume that $G$ can be realized as a graph like if $G$ is the Barycentric
refinement of an arbitrary complex. 
An {\bf edge refinement} of a graph cuts an edge $e=(a,b)$ into two by adding a new vertex $c$ in the
middle and connecting the new vertex to the intersection of spheres at $a$ and $b$.
More formally, we remove the edge $(a,b)$, and adding new edges $(a,c),(c,b)$
as well as $\{ (c,z) \; | \; z \in S(a) \cap S(b) \}$. Edge refinements preserve 
discrete manifolds. More generally:

\begin{propo}
If $G$ is in $\mathcal{X}_d$ and $e$ is an edge in $G$, then the edge refinement is in $\mathcal{X}_d$. 
\end{propo}
\begin{proof}
The effect of the operation on the $f$-vector can be split into two parts.
The first one is to increase $f_0$ and $f_1$ by $1$ (which means adding $t+t^2$ to $f_G(t)$.
Then we add $t f_{S(a) \cap S(b)} + 2 t^2 f_{S(a) \cap S(b)}$, because every
$k$-simplex in $S(a) \cap S(b)$ defines a new $(k+1)$-simplex connecting to $c$ and
two new $(k+2)$-simplices connecting $S(a) \cap S(b)$ to $(a,c)$ and $(b,c)$.
Now, the set of functions satisfying the Dehn-Sommerville symmetry form a linear
space. The claim follows as the added part 
$t+t^2 +t f_{S(a) \cap S(b)} + 2 t^2 f_{S(a) \cap S(b)}$
satisfy the Dehn-Sommerville symmetry by induction because the space $S(a) \cap S(b)$ 
is in $\mathcal{X}_{d-2}$  if $G \in \mathcal{X}_d$. 
\end{proof}

\section{Remarks}

\paragraph{}
For Dehn-Sommerville, see chapter nine in \cite{gruenbaum} for convex polytopes.
It is also covered in \cite{BergerLadder} where we read: 
{\it these relations had already been found by Dehn by 1905 for the
dimensions 3,4,5; they were known in all dimensions by Sommerville by 1927 
but were then forgotten until they were rediscovered by Klee in 1963.}. 
The relations have been extended to larger classes of polytopes. An example of
recent work is \cite{BayerBillera}. More literature is
\cite{Klee1964,NovikSwartz,MuraiNovik,LuzonMoron,BrentiWelker,Hetyei,Klain2002}.

\paragraph{}
The {\bf Levitt curvature} for Euler characteristic 
Formula~(\ref{levitt}) appeared in \cite{Levitt1992}. We have rediscovered that formula
$\chi(G) = \sum_x K(x)$ in the introduction to \cite{cherngaussbonnet}, an article which 
focused on geometric graphs (discrete manifolds). It surprises that {\bf higher dimensional curvature 
in the discrete is so elegant}, especially if one compares to the continuum, where one has to refer to
Pfaffians of curvature expressions to get to the general Gauss-Bonnet-Chern theorem 
(see \cite{Cycon}). In the continuum, the {\bf Euler curvature} is not even defined for odd-dimensional manifolds. 
In the continuum, it is zero for odd-dimensional manifolds as we see here again as it is then a special case of
the Dehn-Sommerville equations. 

\paragraph{}
The Taylor expansion of the parametrized Gauss-Bonnet formula at $t=0$ gives
a {\bf generalized handshake formula} $f_k(G) = \sum_{x \in G} V_{k-1}(S(x))/(k+1)$
which by linearity produces Gauss-Bonnet formulas for any discrete
valuation $X(G) = \sum_k X_k f_k(G)$ and especially for Euler characteristic
$\chi(G) = \sum_k (-1)^k f_k(G)$. One can also just define
$f(t) = 1+\sum_{k=0}^{d} X_k f_k t^{k+1}$ and its anti derivative. The
Gauss-Bonnet formula is the same.  For example, for $X(G) = v_1(G)$, where
$f(t) = 1+v_1 t^2$, the curvature is $K(x) = {\rm deg}(v)/2$. Gauss-Bonnet is then 
{\bf Euler handshake formula}, the {\bf fundamental theorem of graph theory}.
More generally, we have for $v_k(G)$ the curvature is $K(x) = f_{k-1}(S(x))/(k+1)$.

\paragraph{}
In \cite{cherngaussbonnet} we first noticed experimentally that the curvature is zero for
odd-dimensional geometric complexes but we could not prove it yet at that time. 
These zero curvature relations were later proven with
discrete integro-geometric methods in \cite{indexexpectation,colorcurvature} by seeing
curvature as an average of Poincar\'e-Hopf indices when integrating over all functions
or colorings. The connection to Dehn-Sommerville emerged especially in the work 
about {\bf Wu characteristic} \cite{valuation}. So, Dehn-Sommerville conditions appeared
for us three times independently: first as a zero curvature condition for odd-dimensional 
discrete manifolds, then as Barycentric invariants (eigenvectors of $A^T$), then as a 
symmetry for the roots of simplex generating function $t \to f_G(t)$. In 
each case, we were unaware of the Dehn-sommerville connection at first. We hope that this
note makes clear how all these concepts (curvature, Barycentric refinement and root symmetry)
are related. 

\paragraph{}
The classical {\bf Dehn-Sommerville valuations} are
$$ X_{k,d} = \sum_{j=k}^{d-1} (-1)^{j+d} \B{j+1}{k+1} v_j(G) + v_k(G)  \; . $$
If the vectors $X_{0,d}, \dots, X_{d-2,d}$ are written as row vectors in a matrix $X_d$, we have
$$ X_2=\left[\begin{array}{cc} 2 & -2 \\ \end{array} \right],
   X_3=\left[\begin{array}{ccc} 0 & 2 & -3 \\ 0 & 2 & -3 \\ \end{array} \right], 
   X_4=\left[\begin{array}{cccc} 2 & -2 & 3 & -4 \\ 0 & 0 & 3 & -6 \\ 0 & 0 & 2 & -4 \\ \end{array} \right] \; . $$

We have mentioned  before (like \cite{AmazingWorld} that the curvature of 
$X_{k,d}$ is $K(x) = X_{k-1,d-1}(S(x))$. But it is
less obvious there. The reason is the combinatorial identity
$$   X_{k+1,d+1}(l+1)/(l+1) = X(k,d)(l)/(k+2)  \; .  $$
But it also implies that the Dehn-Sommerville curvatures are all zero for a geometric graph.
Use Gauss-Bonnet and induction using the fact that the unit sphere of a geometric
graph is geometric and that for $d=1$, a geometric graph is a cyclic graph $C_n$
with $n \geq 4$. For such a graph, the Dehn-Sommerville valuations are zero.

\paragraph{}
Gauss-Bonnet and Dehn-Sommerville can be generalized to multi-valuate valuations 
like {\bf Wu characteristic} $\omega(G) = \sum_{x \sim y}  \omega(x) \omega(y)$ with 
$\omega(x)=(-1)^{{\rm dim}(x)}$.  The Wu characteristic is then $1-f(-1,-1)$ where 
$$ f(t,s) = 1+\sum_{k,l} f_{kl}(G) t^{k+1} s^{l+1} $$
is the {\bf multivariate simplex generating function}. Here,
$f_{kl}(G)$ is the {\bf $f$-matrix}, counting the number of intersecting $k$-dimensional 
and $l$-dimensional simplices.

\paragraph{}
The curvature of Wu characteristic is then $F_G(t,s) = \int_0^t f(r,s) \; dr$. 
Gauss-Bonnet reads
$$ f_G(t,s) = \sum_{x \in G} F_{S(x)}(t,s)  \;  $$
and especially $\omega(G) = \sum_{x \in G} K(x)$, where $K(x)$ is the Wu curvature.

\paragraph{}
While investigating Barycentric limits \cite{KnillBarycentric2,KnillBarycentric}, 
an other angle to Dehn-Sommerville appeared. 
We first did not see the connection between Barycentric invariants and Dehn-Sommerville. 
The Barycentric refinement operator $A_d$ was first explored empirically by looking at the best 
linear operator implementing the map $f(G) \to f(G_1)$ (brute force data fitting with hundreds
of random graphs) and were surprised that the fitting would lead to an exact formula. 
After getting the formula for $A$ and proving it, we learned that it is ``well known". It 
appears in \cite{Stanley86,LuzonMoron,Hetyei}.

\paragraph{}
The value $g(x,x)=1-\chi((x))$ is the {\bf Green function}, the diagonal entries
of the inverse $g=L^{-1}$ of the unimodular connection matrix $L$ defined as $L(x,y) =1$
if $x \cap y \neq \emptyset$ and $L(x,y)=0$ else. The Green function entries $g(x,y)$
are potential energy values between two simplices $x,y$. We called
$f_G'(-1) = \sum_{x \in G} (1-\chi(S(x))) = {\rm tr}(L-L^{-1})$
the {\bf Hydrogen functional}. 

\paragraph{}
The {\bf energy theorem} assures that the total potential energy $\sum_{x,y} g(x,y)$ 
is the Euler characteristic $\chi(G)$, which is defined as the {\bf super trace}
${\rm str}(L)=\sum_x \omega(x) L(x,x)$ and agrees with
the super trace $\sum_x \omega(x) g(x,x)$ of $g=L^{-1}$. 
The entries $\omega(x) L(x,x)=(-1)^{{\rm dim}(x0)}$
and $\omega(x) g(x,x)$ are integers, the Poincar\'e-Hopf indices \cite{poincarehopf}
of the function $h(x)={\rm dim}(x)$ or $h(x)=-{\rm dim}(x)$ which are 
{\bf colorings} of the graph $\Gamma(G)$.

\paragraph{}
From the energy theorem and Gauss Bonnet we can express $d/dt \log(f_G(t))$ at $t=-1$
through the connection operator $L$. Let $E$ be the matrix which has everywhere $1$.
$\frac{d}{dt} \log(f(t))_{t=-1} = {\rm tr}(L^{-1})/{\rm Tr}(L^{-1}) E)$.
Proof: From Gauss Bonnet, we have
$\frac{d}{dt} \log(f_G(t)) = \frac{f_G'}{f_G} = \sum_x \frac{f_{S(x)}(t)}{f_G(t)}$.
For $t=-1$, we have $\sum_x \chi(S(x))/\chi(G)  = tr(L^{-1})/\chi(G) = {\rm tr}(L^{-1})/{\rm Tr}(L^{-1} E)$.

\paragraph{}
The involutive symmetry $T(f) = \pm f(-1-t)$ given by the Dehn-Sommerville condition
implies {\bf root pairing} for $f$. This article has started
with such an observation. We noticed that for even-dimensional spheres, there is always root with
${\rm Re}(t)=-1/2$ and that the roots are reflection symmetric with respect to $t=-1/2$. 
A simplicial complex is defined to be a {\bf $d$-sphere} if every unit sphere is a
$(d-1)$-sphere and removing one vertex renders the complex contractible.
This inductive definition is primed by the empty complex $0$ being the
$(-1)$-sphere. There are various operations which preserve d-spheres.
We observe that for spheres the roots of $f$ pair up to $-1$ in the odd-dimensional
case and do so also in the even-dimensional case with the remaining roots.

\paragraph{}
There are simplicial complexes outside $\mathcal{X}$ which are Dehn-Sommerville. Similarly 
as zero Euler curvature implies zero Euler characteristic for even-dimensional manifolds,
zero Euler characteristic does not necessarily mean zero curvature. 
The zero Dehn-Sommerville curvature condition is {\bf sufficient} for the complex to be 
Dehn-Sommerville, but it is {\bf not necessary}. There are complexes
which are Dehn-Sommerville, but which are {\bf not Dehn-Sommerville flat}. We give an example
in the illustration section. 

\section{Examples}

\paragraph{}
{\bf Examples.} For $d=3$, we have $A=\left[ \begin{array}{cccc} 1 & 1 & 1 & 1 \\
 0 & 2 & 6 & 14 \\ 0 & 0 & 6 & 36 \\ 0 & 0 & 0 & 24 \\ \end{array} \right]$.
The eigenvectors are $[0, 0, 0, 1]^T, [0, 0, -1, 2]^T, [0, 22, -33, 40]^T, [-1, 1, -1, 1]^T$.
The eigenvector $[-1,1,-1,1]$ to the eigenvalue $1$ is the Euler characteristic. The second
and last eigenvector leads to the Dehn-Somerville invariants $f_2-2f_3=0$ and $f_0-f_1+f_2-f_3=0$.  \\
For $d=4$, we have $A=\left[ \begin{array}{ccccc} 1 & 1 & 1 & 1 & 1 \\ 0 & 2 & 6 & 14 & 30 \\
0 & 0 & 6 & 36 & 150 \\ 0 & 0 & 0 & 24 & 240 \\ 0 & 0 & 0 & 0 & 120 \\ \end{array} \right]$. From
the eigenvectors $[0, 0, 0, 0, 1]$, $[0, 0, 0, -2, 5]$, $[0, 0, 19, -38, 55]$, $[0, -22, 33, -40, 45]$,
$[1, -1, 1, -1, 1]$, the second and last produce Dehn-Sommerville invariants:
$2 v_3-5v_4=-0$ and $22 v_2-33 v_3-40 v_4+45 v_5=0$.

\paragraph{}
For $d=0$, where $h=-1+t+v_0$, the condition is $v_0=2$ implying that the zero-dimensional
graph must have 2 vertices. \\
For $d=1$, where $f=(v_0,v_1)$ gives the number of vertices,
edges and triangles, then $f=1+v_0 t + v_1 t^2$ and $h=(1-v_0+v_1)  + (v_0-2) t + t^2$.
The Dehn-Sommerville condition is $v_0=v_1$ meaning $\chi(G)=0$. Note that we can attach hairs
and even arbitrary trees to a circular graph and still have $v_0=v_1$ satisfied. This shows, that
at least in one dimension, the Dehn-Sommerville relations can hold for a larger class of complexes
than $\mathcal{X}$ to be defined below. \\
For $d=2$, the conditions give $v_1=3v_0-6,v_2=-4+2v_0$. This is equivalent to
$v_0-v_1+v_2=2, 2v_1=3v_2$. This means that Euler characteristic is $2$ and that every edge meets two
triangles. \\
For $d=3$, the condition is equivalent to $\chi(G)=v_0-v_1+v_2-v_3=0$
and $22v_1-33v_2+40v_3=0$. We will see in a moment how to get the more intuitive
Barycentric expressions through eigenvectors of the Barycentric refinement operators.

\paragraph{}
The root pairing property was already mentioned in \cite{averagedimension}. We found this while investigating
the statistics of the simplex cardinality distribution in simplicial complexes. 
The {\bf root pairing statement} is obviously true for $0$-dimensional spheres. If a zero dimensional complex
has $n$ points, then the generating function of the $f$-vector is $1+nt$. This has a root $-1/2$
if and only the complex has exactly $n=2$ point, which means that $G$ has to be a $0$-sphere. 
Let us also mention the $(-1)$-dimensional complex which is the empty complex. In that case, 
the function is $f=1$ which has no roots. Root pairing still works, there are just no pairs. 

\paragraph{}
For $1$-dimensional complexes with $n$ vertices and $m$ edges, we have the generating
function $1+n t + m t^2$. The Euler characteristic is $n-m = \chi(G)$. The roots
are $-n \pm \sqrt{-4m + n^2})/(2m)$. The sum of the roots is $-n/m$. This is $-1$ 
if and only if $n=m$, meaning that we need $\chi(G)=0$. Beside circular complexes, there are
many complexes like {\bf sun graphs} for which $n=m$. We can attach arbitrary trees to the
circular graph for example and still have the property. There is a sphere complex which is
not the Whitney complex of a graph, which is $G=\{ \{1,2\},\{2,3\},\{3,1\},\{1\},\{2\},\{3\} \}$
where $n=3,m=3$ and where the roots become complex. We see confirmed here that roots are
not real if and only if the complex is not the Whitney complex of a graph. 

\paragraph{}
For $2$-dimensional complexes the simplex generating function is 
$1+n t+m t^2+l t^3$, in order to have a root $-1/2$, we need 
$n=(8-l+2m)/4$. For a $2$-manifold, we have $2m=3l$.
The two equations give $m=3n-6, l=2n-4$ implying $\chi(G)=n-m+l=2$. 
Actually, for two dimensional complexes, the two equations $2m=3l, \chi(G)=n-m+l=2$ 
imply that $f(-1/2)=0$. This in particular holds for two disjoint copies of the
projective plane. 


\pagebreak

\section{Illustrations}

\begin{figure}[!htpb]
\scalebox{0.971}{\includegraphics{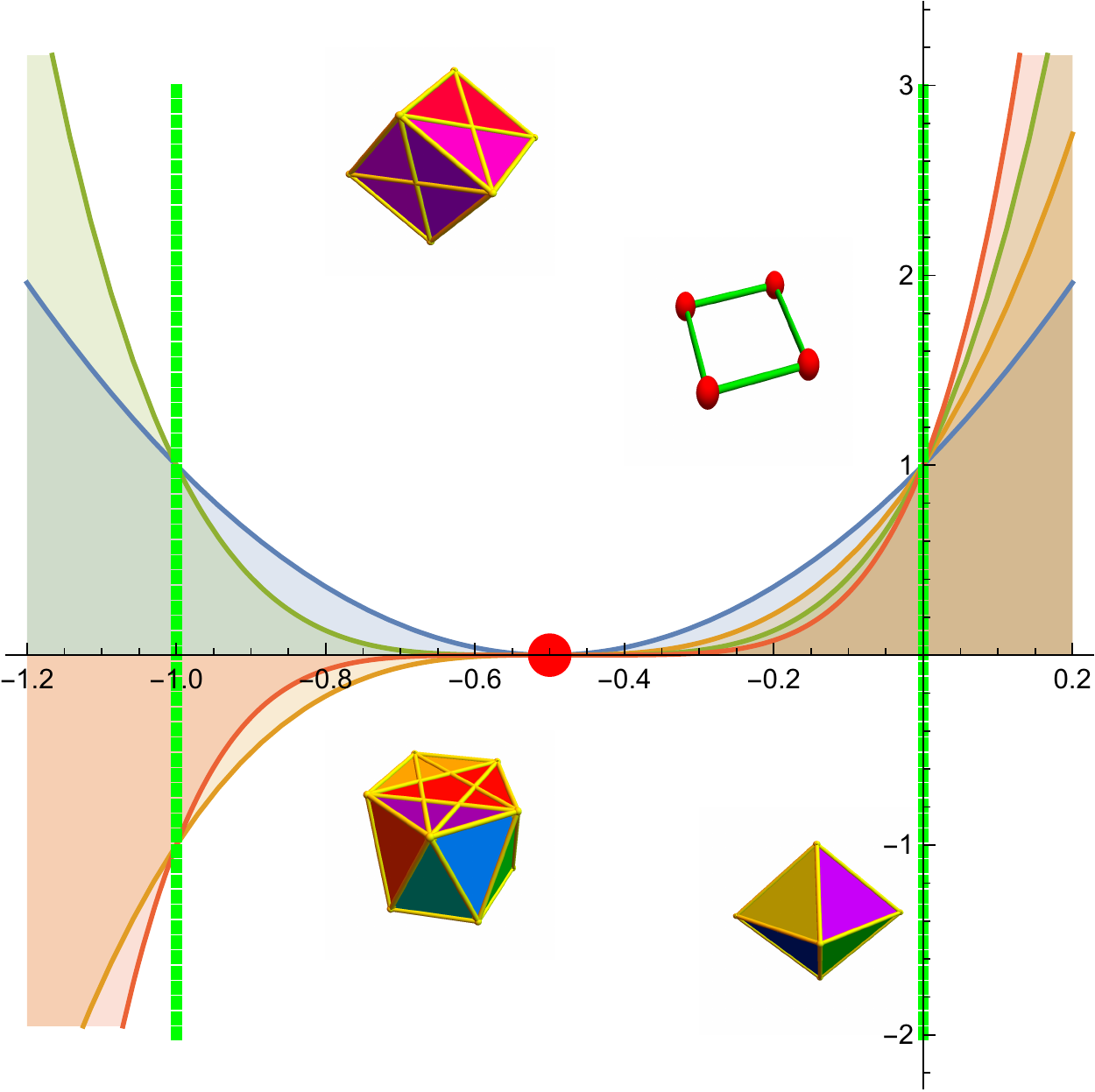}}
\label{spheres}
\caption{
The functions $f_G$ for the smallest spheres $S^1=C_4=S^0+S^0$, $S^2 = C_4+S^0=S^0+S^0+S^0$
(the octahedron), $S^3 = S^2 + S^0 = S^1 + S^2$ (the three sphere), $S^4=S^4 + S^0 = 5*S^0$
(the four sphere), which are all cross polytopes. The index generating function $f_G(t)$ of $G=S^0$ is $1+2t$.
so that $f_{S^d}(t) = (1+2t)^{d+1}$. We then observed experimentally that all spheres satisfy the 
symmetry $f(x)+(-1)^d f(-1-x)=0$, then linked it to Dehn-Sommerville. 
}
\end{figure}

\begin{figure}[!htpb]
\scalebox{0.971}{\includegraphics{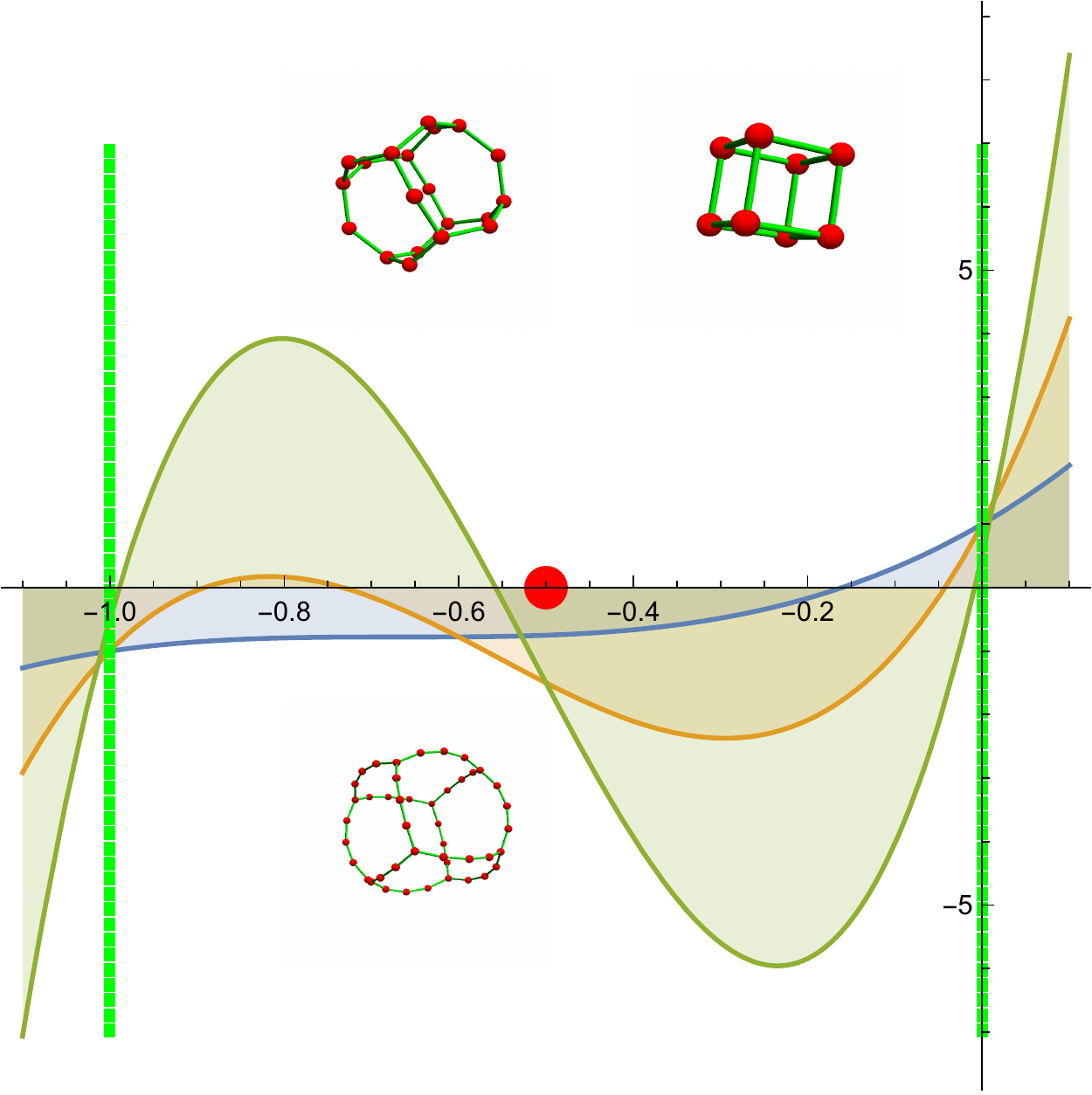}}
\label{cube}
\caption{
If the faces are included to a cube graph we get a {\bf CW-complex} which models a discrete 2-sphere.
Its generating function is $f_G(x)=1 + 8 x + 12 x^2 + 6 x^3$. It does not satisfy Dehn-Sommerville. 
It also has non-real roots. After Barycentric refinements however, the roots become real. 
We see $f_{G_1}(x)=1 + 26 x + 60 x^2 + 36 x^3$ and $f_{G_2}(x)=1+122 x + 336 x^2 + 216 x^3$ (we 
plotted $f_{G_2}/2$). The coefficients $[122,336,216]$ are already aligned quite well with the
Perron-Frobenius eigenvector $[1,3,2]^T$ to the Barycentric refinement operator 
$A_2$ in dimension
$2$ which defines a function having only real roots. In general we see that the Perron-Frobenius
functions $a_1 t + \cdots + a_n t^{d+1}$ for the Perron-Frobenius eigenvector to the $(d+1) \times (d+1)$
matrix $A_d$ always has only real roots. It looks like a simple calculus/linear algebra problem, but we
can not prove this yet. It would imply that for sufficiently large Barycentric refinement of any
CW complex and especially simplicial complexes, the roots of $f_{G_n}$ are real for large enough $n$. 
}
\end{figure}

\begin{figure}[!htpb]
\scalebox{0.971}{\includegraphics{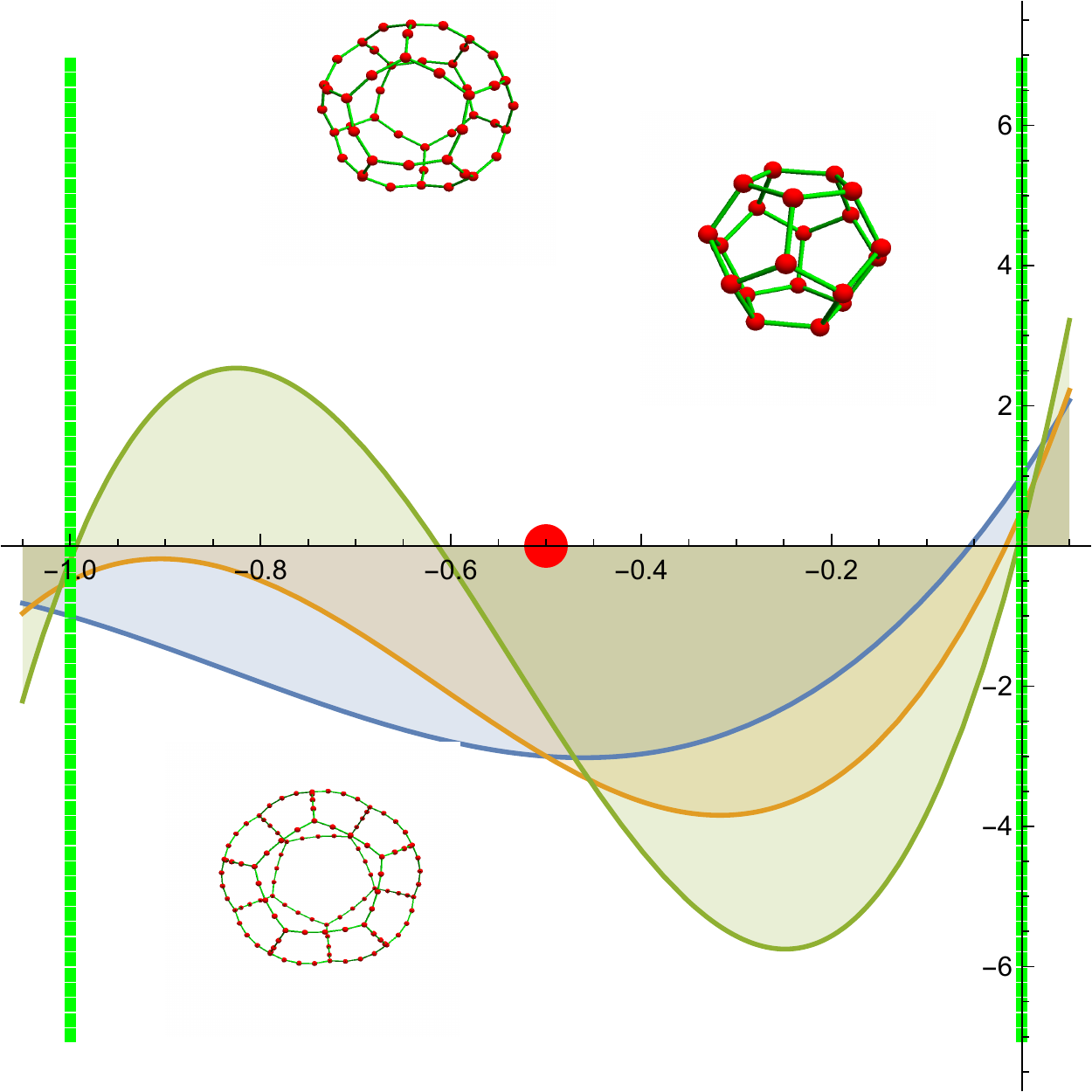}}
\label{cube}
\caption{
Also the {\bf dodecahedron} (when seen as a CW-sphere and not a 1-dimensional graph, which it is when
seen as a simplicial complex), has non-real roots for $f_G(t) = 1+20t+30t^2+12t^3$. But here also $f_{G_1}$ has non-real
roots. Only $f_{G_2}$ for the second Barycentric refinement $G_2$ starts to have real roots. As the Perron-Frobenius eigenvector produces a 
function $f$ which satisfies the Dehn-Sommerville symmetry, we get roots for $f_{G_n}$ which are more and more symmetric. 
Also this was just observed experimentally at first. The linear algebra of the eigenvectors of the 
Barycentric refinement operators $A_d$ explains this. Indeed, as we show here, Dehn-Sommerville
for complexes of the type $\mathcal{X}_d$ is a manifestation for a symmetry one has in the 
Barycentric limit. Since the involutary symmetry (duality) in the limit has eigenvalues $1$
or $-1$, only half of the Barycentric invariants are Dehn-Sommerville invariants.  }
\end{figure}

\begin{figure}[!htpb]
\scalebox{0.971}{\includegraphics{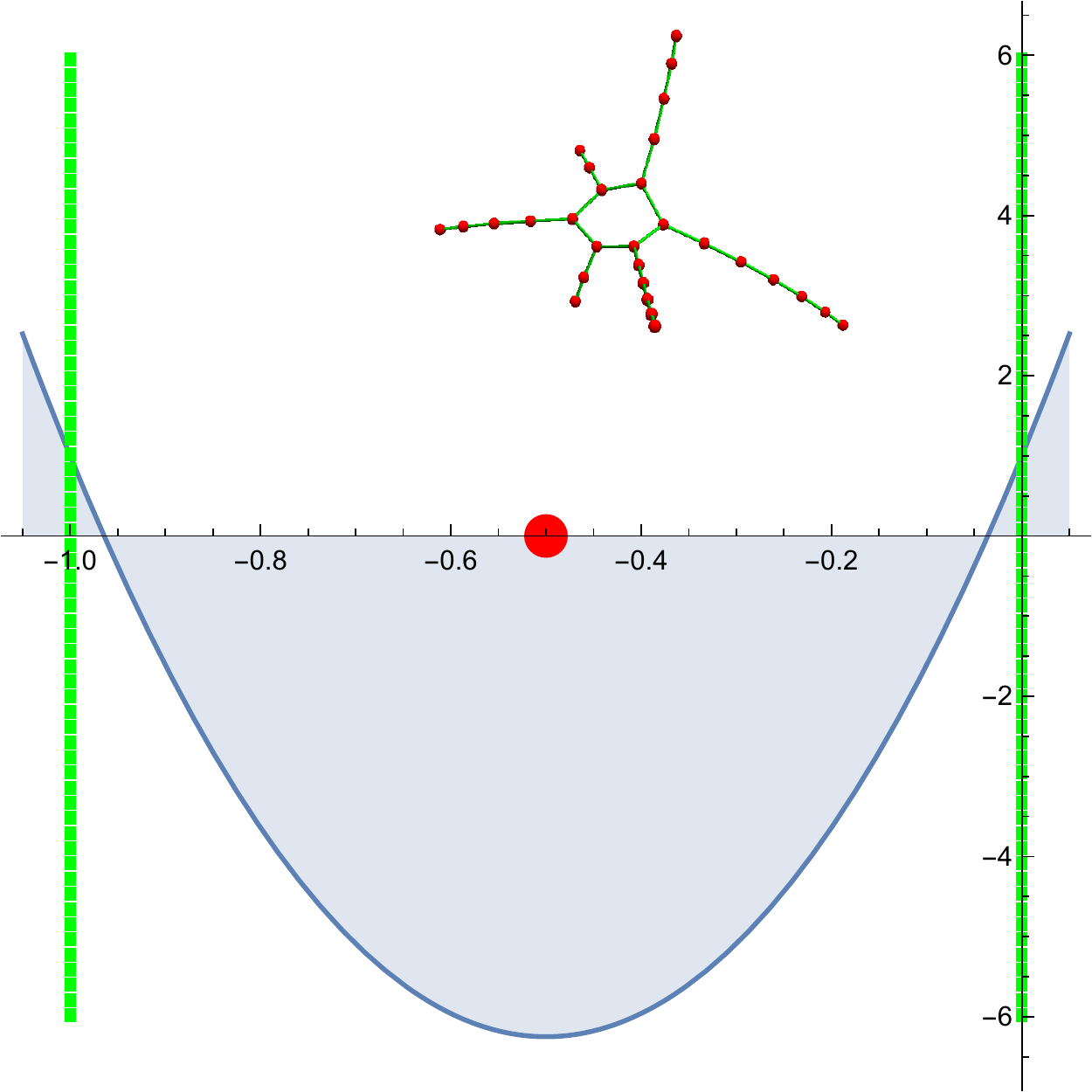}}
\label{sun}
\caption{
{\bf Sun graphs} are $1$-dimensional complexes which satisfy Dehn-Sommerville, even-so they 
are only varieties, not manifolds. As they have the same number of vertices than edges, we have 
$f(t) = 1+n t + n t^2$ which satisfies $f(-1-t)=f(t)$. It is an example of a $1$-variety. 
A $d$-variety is a complex $G$ for which all unit spheres
are $d-1$ varieties. Like manifolds, it starts with the induction that the empty complex
is a $-1$ variety but unlike for manifolds, we do not insist that unit spheres are 
$(d-1)$-spheres. In this example, $f_G(t) = 1 + 29t + 29t^2$. The roots $-1/2 \pm  \sqrt{25/116}$ 
are symmetric with respect to ${\rm Re}(t)=-1/2$. We have Dehn-Sommerville symmetry. 
}
\end{figure}

\begin{figure}[!htpb]
\scalebox{0.971}{\includegraphics{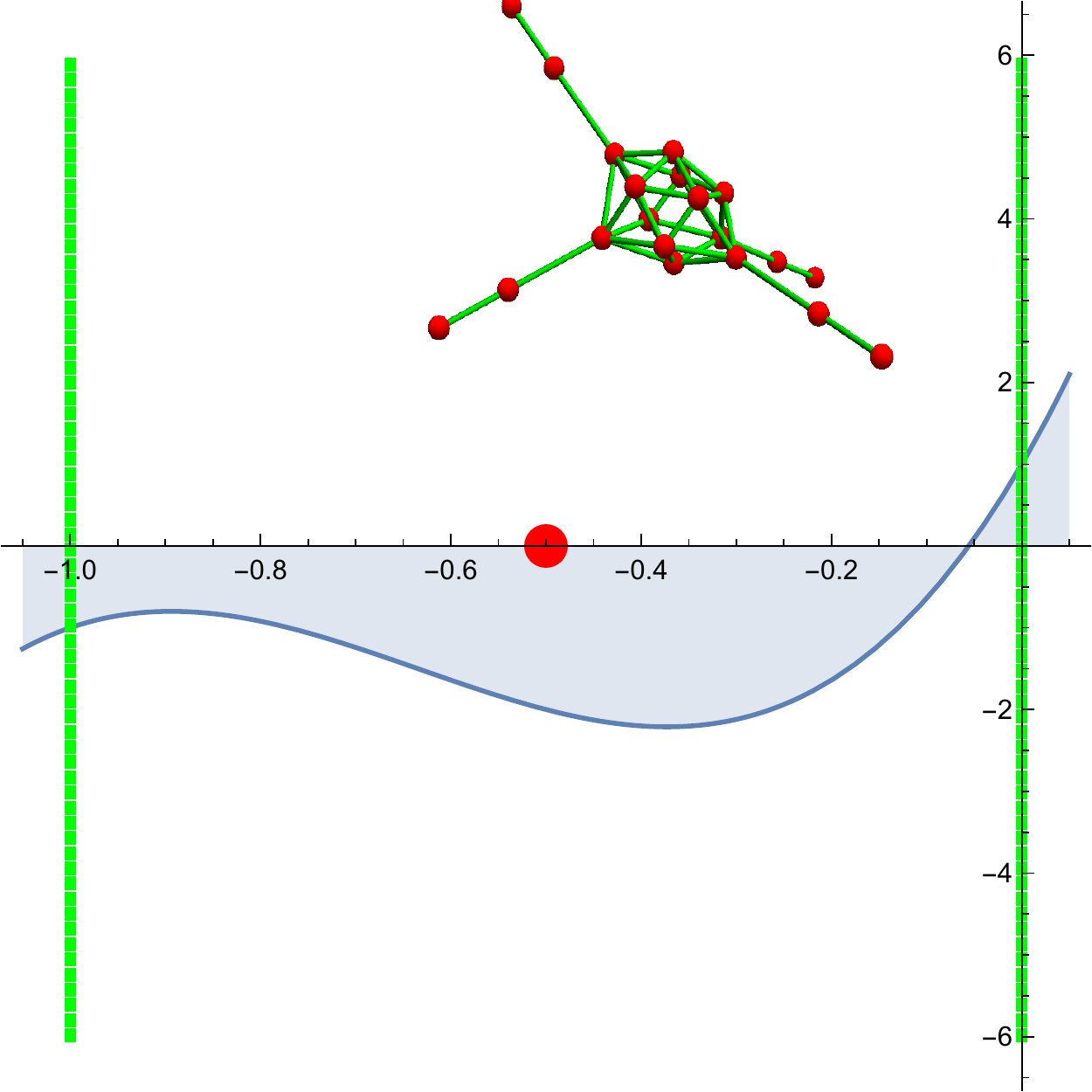}}
\label{hairs}
\caption{
When adding hairs to a $2$-sphere, the Dehn-Sommerville property gets destroyed. 
The complex $G$ shown here is a simplicial complex but it is not pure. Its inductive
dimension is $47/30=1.56667$. Its average simplex cardinality $f_G'(1)/f_G(1)$ is
$156/79=1.97468...$ for the function $f_G(t) = 1 + 20t + 38t^2 + 20t^3$. The function 
$f_G$ does not honor the Dehn-Sommerville symmetry $f(t)=\pm f(-1-t)$. We have
$f(-1-t) = -1-4t-22t^2-20t^3$.  What happens is that the unit spheres do not satisfy
Dehn-Sommerville. There are unit spheres which are a disjoint union of a 1-sphere and
a point which does not satisfy Dehn-Sommerville. This means that the Dehn-Sommerville
curvatures are not zero. The complex is not Dehn-Sommerville flat. 
}
\end{figure}

\begin{figure}[!htpb]
\scalebox{0.971}{\includegraphics{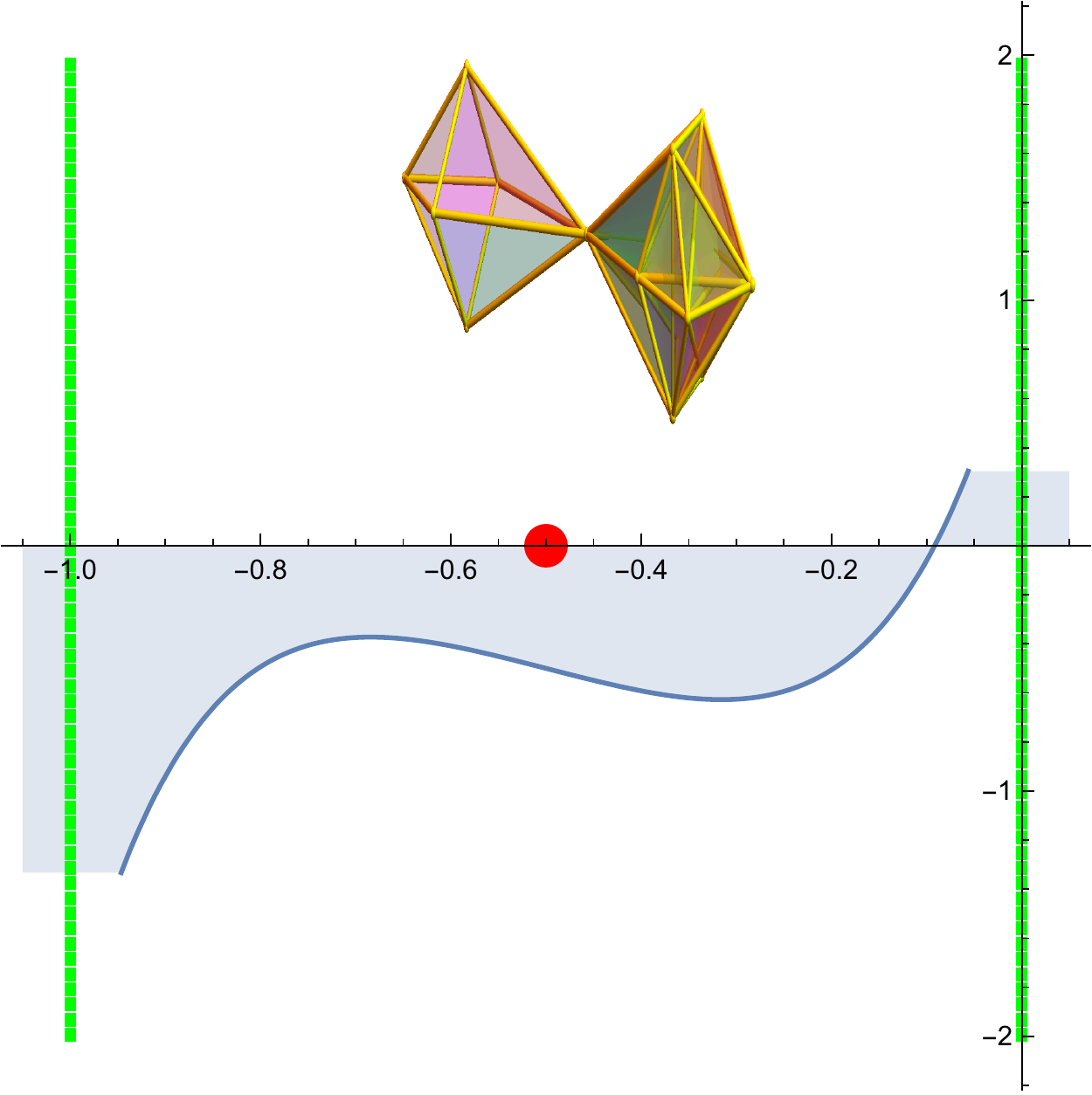}}
\label{connectedsum}
\caption{
The Dehn-Sommerville property gets distroyed with disjoint sums as
well as most connected sums. We see here the connected sum $G$ of a 2-sphere $O$ (an 
octahedron graph) and a $4$-sphere $C_4 + O$ joined at a vertex $v$. 
The unit sphere $S(v)$ is a disjoint union of a 1-sphere $C_4$ and a 3-sphere $C_4 + C_4$. 
This disjoint union does not satisfy Dehn-Sommerville. By
Gauss-Bonnet (since all other points are Dehn-Sommerville flat), also $G$ is not
Dehn-Sommerville. It is almost, $f+1/2$ would satisfy the Dehn-Sommerville symmetry. 
}
\end{figure}

\begin{figure}[!htpb]
\scalebox{0.971}{\includegraphics{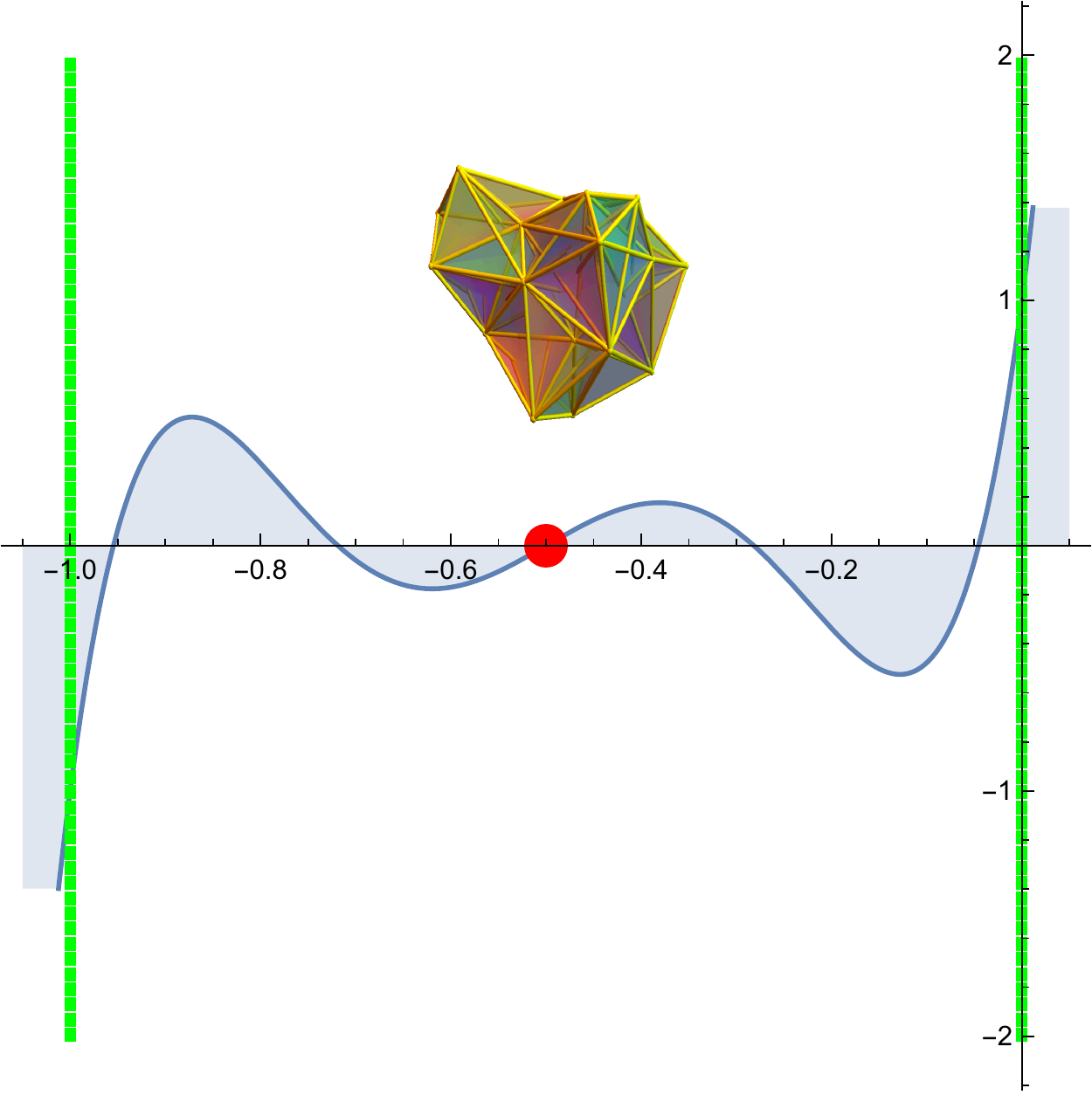}}
\label{randomfoursphere}
\caption{
We see a random four sphere $G$. It is Dehn-Sommville of course. As for 
any even dimensional sphere, there is a root $t=-1/2$ for the 
simplex generating function $f_G(t)$. 
}
\end{figure}

\begin{figure}[!htpb]
\scalebox{0.971}{\includegraphics{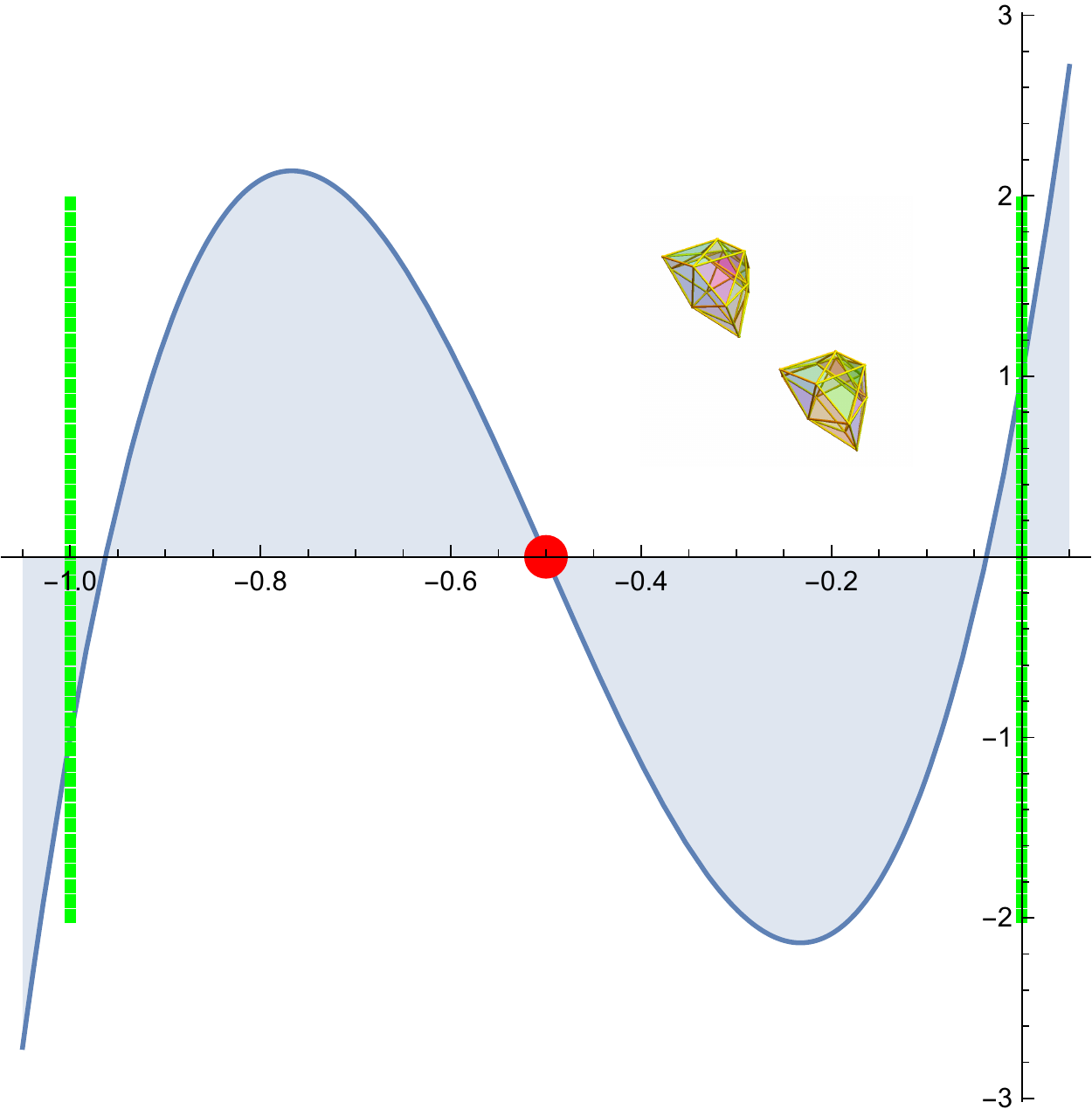}}
\label{projectiveunion}
\caption{
We see the disjoint union of two projective planes. Any even dimensional
manifold of Euler characteristic $2$ satisfies the Dehn-Sommerville condition. 
So also $G=\mathbb{P}^2 \cup \mathbb{P}^2$.  We have $f_G(t) = (1 + 2x)*(1 + 28x + 28x^2)$. 
The Betti-vector is $(b_0,b_1,b_2) = (2,0,0)$. Obviously, Poincar\'e duality is failing 
for $G$ as $G$ is non-orientable. Still, Dehn-Sommerville is intact. 
}
\end{figure}

\begin{figure}[!htpb]
\scalebox{0.971}{\includegraphics{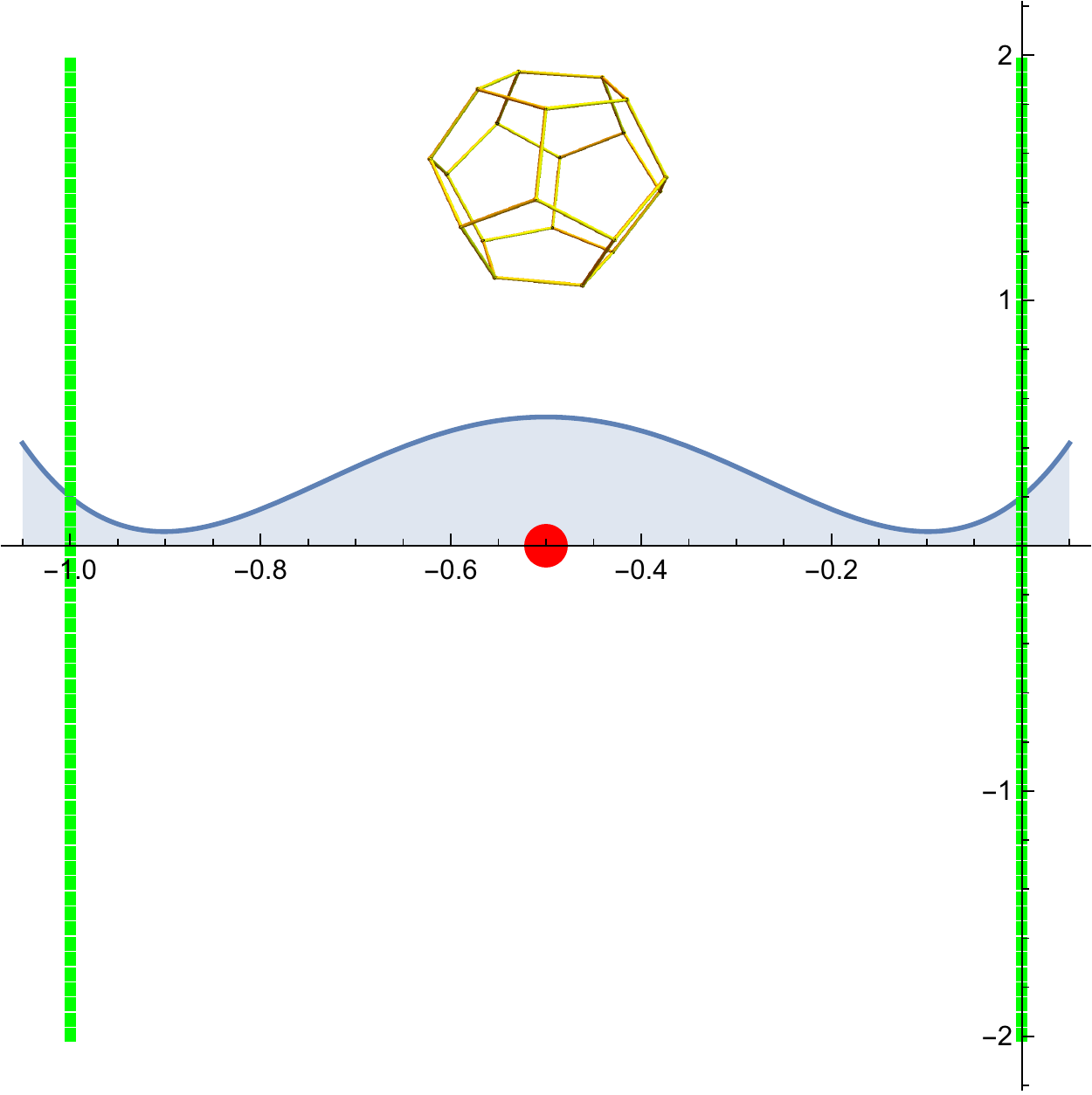}}
\label{poincaresphere}
\caption{
We see the graph of $f_G(t) = 1+16t + 106 t^2+180 t^3+90 t^4$, where $G$ is a {\bf Poincar\'e sphere complex}
with $16$ zero-dimensional simplices, found in \cite{BjoernerLutz}. All 4 roots of $f_G$ are complex. 
As a 3-manifold with zero Euler characteristic, $G$ must be Dehn-Sommerville.
}
\end{figure}

\begin{figure}[!htpb]
\scalebox{0.971}{\includegraphics{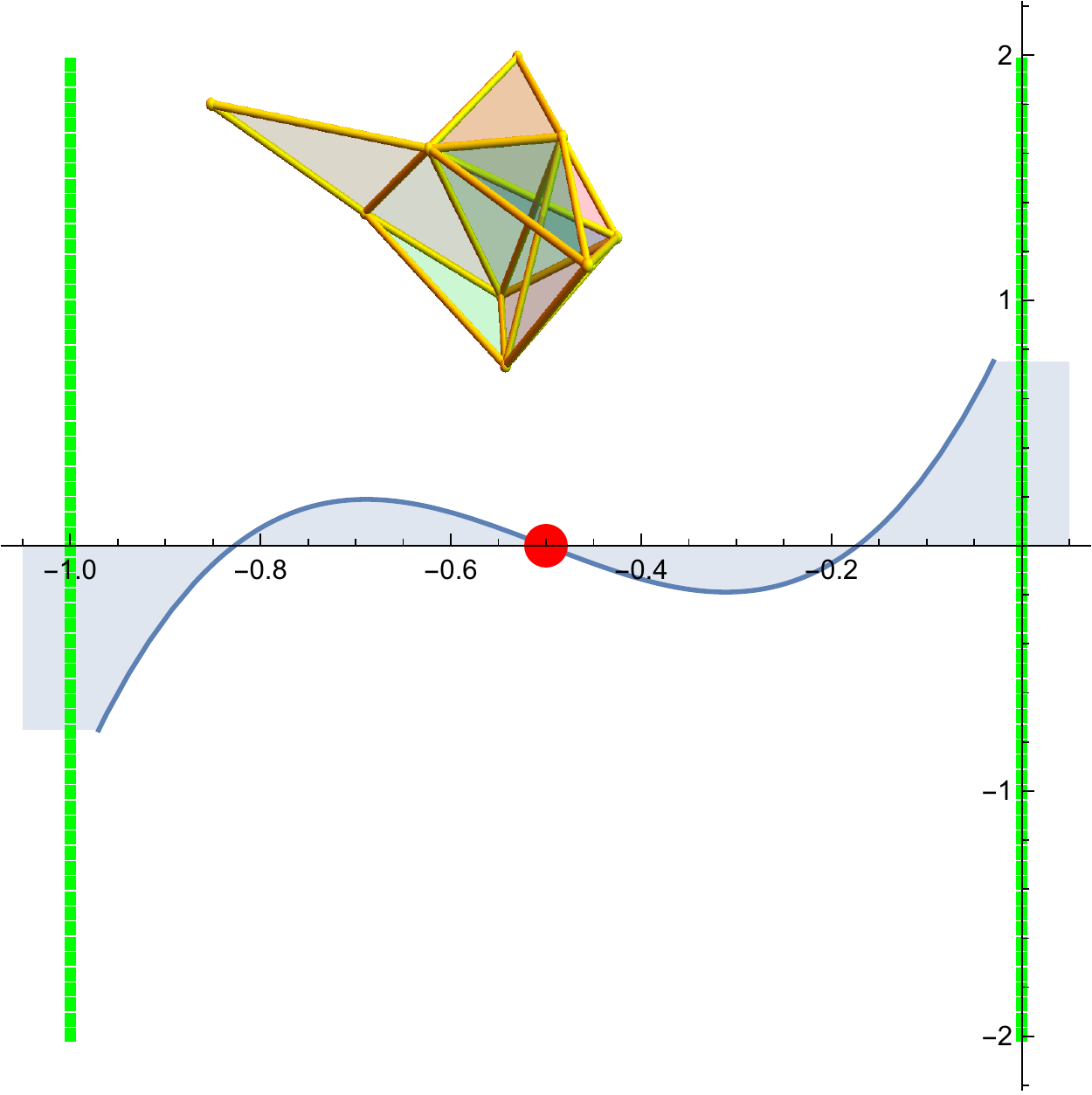}}
\label{poincaresphere}
\caption{
We see a graph obtained by poking around randomly in Erdoes-Renyi spaces
looking for Dehn-Sommerville graphs. This example has the simplex generating function
$f_G(t) = 1 + 9t + 21t^2 + 14t^3$ but it is not a 2-sphere. It is not a manifold but 
has the Betti numbers $(b_0,b_1,b_2)=(1,0,1)$ of the $2$-sphere. 
The graph has inductive dimension $2$. It is interesting
for us because it is an example which is not Dehn-Sommerville flat. It shows that there
are Dehn-Sommerville complexes for which some unit spheres are not Dehn-Sommerville. 
The complex is {\bf Dehn-Sommerville non-flat}. In other words, having zero Dehn-Sommerville
curvatures (unit spheres are Dehn-Sommerville) is only sufficient and not necessary for
$G$ to be Dehn-Sommerville. 
}
\end{figure}

\section{Code}

\paragraph{}
Here is Mathematica code (see ArXiv version to copy paste) which computes 
$F_{S(x)}(t)$, then adds up to $f_G(t)$. 

\begin{small}
\lstset{language=Mathematica} \lstset{frameround=fttt}
\begin{lstlisting}[frame=single]
UnitSphere[s_,a_]:=Module[{b=NeighborhoodGraph[s,a]},
  If[Length[VertexList[b]]<2,Graph[{}],VertexDelete[b,a]]];
UnitSpheres[s_]:=Module[{v=VertexList[s]},
  Table[UnitSphere[s,v[[k]]],{k,Length[v]}]];
ErdoesRenyi[M_,p_]:=Module[{q,e,a},V=Range[M];
  e=EdgeRules[CompleteGraph[M]]; q={};
  Do[If[Random[]<p,q=Append[q,e[[j]]]],{j,Length[e]}];
  UndirectedGraph[Graph[V,q]]];
CliqueNumber[s_]:=Length[First[FindClique[s]]];
ListCliques[s_,k_]:=Module[{n,t,m,u,r,V,W,U,l={},L},L=Length;
  VL=VertexList;EL=EdgeList;V=VL[s];W=EL[s]; m=L[W]; n=L[V];
  r=Subsets[V,{k,k}];U=Table[{W[[j,1]],W[[j,2]]},{j,L[W]}];
  If[k==1,l=V,If[k==2,l=U,Do[t=Subgraph[s,r[[j]]];
  If[L[EL[t]]==k(k-1)/2,l=Append[l,VL[t]]],{j,L[r]}]]];l];
Whitney[s_]:=Module[{F,a,u,v,d,V,LC,L=Length},V=VertexList[s];
  d=If[L[V]==0,-1,CliqueNumber[s]];LC=ListCliques;
  If[d>=0,a[x_]:=Table[{x[[k]]},{k,L[x]}];
  F[t_,l_]:=If[l==1,a[LC[t,1]],If[l==0,{},LC[t,l]]];
  u=Delete[Union[Table[F[s,l],{l,0,d}]],1]; v={};
  Do[Do[v=Append[v,u[[m,l]]],{l,L[u[[m]]]}],{m,L[u]}],v={}];v];
Fvector[s_]:=Delete[BinCounts[Map[Length,Whitney[s]]],1];
Ffunction[s_,x_]:=Module[{f=Fvector[s],n},n=Length[f];
  If[Length[VertexList[s]]==0,1,1+Sum[f[[k]]*x^k,{k,n}]]];
DehnSommerville[s_]:=Module[{f},Clear[x];f=Ffunction[s,x];
   Simplify[f] === Simplify[(f /. x->-1-x)]];
Curvature[s_,x_]:=Module[{g=Ffunction[s,y]},
  Integrate[g,{y,0,x}]];
EulerChi[s_]:=Module[{f=Fvector[s]}, 
  -Sum[f[[k]](-1)^k,{k,Length[f]}]]
Curvatures[s_,x_]:=Module[{S=UnitSpheres[s]},
   Table[Curvature[S[[k]],x],{k,Length[S]}]];

s=ErdoesRenyi[16,0.4]; 
{Ffunction[s,x],Curvatures[s,x]}
{EulerChi[s],-Total[Curvatures[s,x]] /. x->-1}

threesphere=UndirectedGraph[Graph[{1->3,1->4,1->5,1->6,1->7,
1->8,3->2,3->5,3->6,3->7,3->8,4->2,4->5,4->6,4->7,4->8,5->2,
5->7,5->8,6->2,6->7,6->8,7->2,8->2}]];
Print[DehnSommerville[threesphere]];
Print[DehnSommerville[StarGraph[10]]];
Print[DehnSommerville[WheelGraph[10]]];
Print[DehnSommerville[CompleteGraph[5]]];
Print[DehnSommerville[CompleteGraph[{3,3}]]];
Print[DehnSommerville[CycleGraph[5]]];
\end{lstlisting}
\end{small}

\paragraph{}
And here are the Barycentric invariants

\begin{small}
\lstset{language=Mathematica} \lstset{frameround=fttt}
\begin{lstlisting}[frame=single]
A[n_]:=Table[StirlingS2[j,i]*i!,{i,n+1},{j,n+1}];
Invariants[n_]:=Eigenvectors[Transpose[A[n]]];
MatrixForm[Transpose[Reverse[Invariants[4]]]]

Fcrosspolytop[n_]:= Delete[CoefficientList[(1+2x)^(n+1),x],1]
Binvariants=Invariants[4];f=Fcrosspolytop[4]; (* 4-sphere *)
Print[f];
Print[MatrixForm[Binvariants]]; 
Binvariants[[2]].f
Binvariants[[4]].f 
\end{lstlisting}
\end{small}

\section{Questions}

\paragraph{}
(A) One open question is: for which complexes are there non-real roots of $f$?

We measure that for all simplicial $G$, the roots of $f_{G_n}$ for 
Barycentric refinements $G_n$ of $G=G_0$ are 
real and contained in the open interval $(-1,0)$ if $n$ is large enough. This is a still unsolved
concrete calculus problem as it requires to find roots of explicitly given polynomials defined by
Perron-Frobenius eigenvectors of the Barycentric refinement operator.

\paragraph{}
We have never seen that non-real roots from $f$ appear after a few Barycentric refinements.
Non-real roots can appear for Whitney complexes of random graphs, for non-Whitney complexes like 
{\bf homology spheres} $G$ like the one with $f=1+16t+106t^2+180t^3+90t^4$ or the 
{\bf Barnette 3-sphere} with $f=1+8t+27t^2+38t^3+19t^4$ or
then the boundary sphere of a simplex like the {\bf tetrahedral sphere} with 
$f=1+4t+6t^2+4t^3$ or sphere-CW-complexes like the {\bf cube} $f=1+8t+12t^2+6t^3$, 
the roots of $f$ can become complex. 

\paragraph{}
(B) An other question is to see how large the set of Dehn-Sommville complexes are if
we look at all graphs with $n$ vertices. For $n=1,2,3$ there are none, for $n=4$, 
there is only the cyclic graph, for  $n=5$, we have only cyclic $C_5$ of $C_4$ with 
a hair. When fishing randomly in the pool of Erd\"os-R\'enyi graphs, we come up empty 
in general. The probability of having a Dehn-Sommville complex must be very small. 
One can wonder whether the probability is exponentially small in $n$ or super exponentially
small in $n$.  

\bibliographystyle{plain}

\end{document}